\documentclass[10pt,leqno]{amsart}
\usepackage{amsmath}
\usepackage{amsfonts}
\usepackage{amssymb}
\usepackage{graphicx}
\usepackage{color}
\usepackage{hyperref}
\usepackage{xcolor}

\newtheorem{theorem}{Theorem}[section]
\newtheorem*{theorem*}{Theorem}
\newtheorem{lemma}{Lemma}[section]

\newtheorem{remark}[theorem]{Remark}

\def\p{\partial}

\def\R{\mathbb{R}}

\def\vp{\varphi}

\def \p {\partial}

\numberwithin{equation}{section}

\begin{document}

\title[Neumann eigenvalues in Gauss space]{An isoperimetric inequality  for lower order Neumann eigenvalues in Gauss space}

\author{Yi Gao} 
\address{School of Mathematical Sciences, Soochow University, Suzhou, 215006, China}
\email{ yigao\underline{~}1@163.com}

\author{Kui Wang} 
\address{School of Mathematical Sciences, Soochow University, Suzhou, 215006, China}
\email{kuiwang@suda.edu.cn}



\begin{abstract}
We prove a sharp isoperimetric inequality for  the harmonic mean of the first $m-1$ nonzero Neumann eigenvalues  for    Lipschitz domains symmetric about the origin in Gauss space. Our result generalizes the Szeg\"o-Weinberger type inequality in Gauss space, which was proven by  Chiacchio and Di Blasio in \cite[Theorem 4.1]{CB12}.

\end{abstract}
\maketitle

\section{Introduction}
The classical Szeg\"o-Weinberger inequality \cite{Sz54, Wei56} states that the ball uniquely maximizes the first nonzero Neumann eigenvalue among bounded domains with the same volume in Euclidean space. This result has been extended to bounded domains in the hemisphere and in hyperbolic space \cite{AB95}. For further results, we refer to \cite{AS96, AN12, As99, BCB16,  LL23, Wang19} and references therein.
Regarding lower order Neumann eigenvalues, Ashbaugh and Benguria \cite{AB93} conjectured that for any bounded Lipschitz domain $\Omega$ in $\R^m$, the following inequality holds:
\begin{align} \label{1.1}
            \frac{1}{\mu_1(\Omega) }+\frac{1}{\mu_2(\Omega) }+\cdots + \frac{1}{\mu_{m}(\Omega) }\geq \frac{m}{\mu_1(B)},
        \end{align}
where $B$ is a round ball of  the same volume as $\Omega$.
It had been noticed by Szeg\"o and Weinberger that Szeg\"o's approach (see \cite{Sz54}) actually yields
an isoperimetric result for the sum of the reciprocals of the first two nonzero Neumann eigenvalues, $\frac 1 {\mu_1(\Omega)}+ \frac 1 {\mu_2(\Omega)}$, where $\Omega\subset\R^2$ is a simply connected domain. 
For planar domains that are not necessarily simply connected, as well as in dimensions  higher than two, the problem remains open.   
 Recently, Xia and  Wang  \cite{XW23} established a sharp isopermetric inequality for the harmonic mean of the first $m-1$ nonzero Neumann eiegnvalues of Laplacian  on any bounded Lipschitz domain in $\R^m$ and $\mathbb{H}^m$, precisely
 \begin{align} \label{1.2}
            \frac{1}{\mu_1(\Omega) }+\frac{1}{\mu_2(\Omega) }+\cdots + \frac{1}{\mu_{m-1}(\Omega) }\geq \frac{m-1}{\mu_1(B)},
        \end{align}
which strongly supports Conjecture \eqref{1.1}.  Wang and Xia \cite{XW23} also proved  \eqref{1.2} in hyperbolic space; Benguria, Brandolini, and Chiacchio  \cite{BBC20} proved \eqref{1.2} in the hemisphere. Additionally, Wang and Xia \cite{WX21} proved a similar inequality to \eqref{1.2} for the ratio of Dirichlet eigenvalues,  and Meng and the second-named author \cite{MW24} derived a similar inequality  in rank-1 symmetric spaces. Chen and Mao \cite{CM24}  studied sharp isoperimetric estimates for the lower-order eigenvalues of the Witten-Laplacian.
For further progress on Ashbaugh and Benguria's conjecture, we refer to  \cite{As99, He06} and the  references therein. 

Eigenvalue optimization problems for  the Laplace operator under the Gaussian measure have also attracted widespread attention. For instance,  Chiacchio and Di Blasio \cite{CB12} proved that, among all origin-symmetric regions with  fixed Gaussain volume, the ball maximizes the first nonzero Neumann eigenvalue. Additional results on eigenvalue problems under Gaussian measure can be found in \cite{AN12, BCHT13, BCKT16, BCB16, CG22} and the references therein.

In this paper, we consider the Neumann eigenvalue problem in Gauss space.  Let $d\gamma_m=(2\pi)^{-m/2}e^{-|x|^2/2}dx$  denote the $m$-dimensional Gaussian measure, and let $\Omega\subset \R^m$ be a connected  Lipschitz domain. We define  $H^1(\Omega, d\gamma_m)$  as the weighted Sobolev space equipped with the norm
\begin{align*}
||u||_{H^1(\Omega, d\gamma_m)}:=\left(\int_\Omega u^2 \, d\gamma_m\right)^{1/2}+ \left(\int_\Omega |\nabla u|^2 \, d\gamma_m\right)^{1/2}. 
\end{align*}
The Neumann eigenvalue problem on  $\Omega$ is given by 
\begin{align}\label{1.3}
    \begin{cases}
      -\Delta u+x\cdot\nabla u
     =\mu  u,& \quad \text{in}\quad \Omega,\\
        \frac{\p u}{\p \nu}=0,& \quad \text{on} \quad \p \Omega,   
    \end{cases}
\end{align}
where $\frac{\p u}{\p \nu}$ denotes the outward normal derivative on $\p \Omega$. It is well known that the problem \eqref{1.3} has discrete eigenvalues, with an abuse of notation still denoted by $\mu_k(\Omega)$ for $k=0,1,\cdots$, which satisfy
  \begin{align*}
       0=\mu_0(\Omega) < \mu_1(\Omega) \leq \mu_2(\Omega) \leq \cdots \to +\infty,
    \end{align*}
with each eigenvalue repeated according to its multiplicity. Here, $\mu_k(\Omega)$ denotes the  $k$th Neumann eigenvalue of \eqref{1.3} in Gauss space, and $u_k(x)$ denotes the corresponding eigenfunction. By standard spectral theory for self-adjoint compact operators,  $\mu_k(\Omega)$ admits the variational characterization: 
\begin{align}\label{1.4}
\mu_k(\Omega)=\inf_{u\in H^1(\Omega, d\gamma_m)}\Big\{\frac{\int_\Omega |\nabla u|^2\, d\gamma_m}{\int_\Omega u^2\, d\gamma_m}: \int_\Omega u u_{j}\, d\gamma_m=0\text{\quad for\quad} 0\le j \le k-1\Big\}.
\end{align}
For Neumann eigenvalues defined in this way, Chiacchio and Di Blasio \cite[Theorem 4.1]{CB12} proved that the Euclidean ball centered at the origin uniquely maximizes $\mu_1(\Omega)$ among all sets $\Omega$ symmetric about the origin with the same Gaussian volume.  In this paper,  we establish a sharp estimate for   the harmonic mean of the first $m-1$ nonzero Neumann eigenvalues in Gauss space. The main result is as follows.
\begin{theorem}\label{thm 1}
Let $\Omega\subset \R^m$ be a  Lipschitz domain (possibly unbounded) symmetric about the origin, and let $B \subset \R^m$  be the  origin-centered ball with the same Gaussian
volume as $\Omega$,   i.e. 
$ \int_{\Omega}\, d\gamma_m= \int_{B}\, d\gamma_m$.
Let $\mu_i(\Omega)$ be the Neumann eigenvalues of \eqref{1.3}.  Then 
\begin{align}\label{1.5}
    \frac{1}{\mu_1(\Omega)}+ \frac{1}{\mu_2(\Omega)}+\cdots+\frac{1}{\mu_{m-1}(\Omega)}\ge \frac{m-1}{\mu_1(B)}.
\end{align}
Equality holds  if and only if 
$\Omega=B$.
\end{theorem}
Note that $\mu_1(\Omega)\le \mu_2(\Omega)\le \cdots\le \mu_{m-1}(\Omega)$, so inequality \eqref{1.5} strengthens the Szeg\"o-Weinberger inequality
$\mu_1(\Omega)\le \mu_1(B)$
proven in \cite[Theorem 4.1]{CB12} in  Gauss space.
The symmetry assumption on $\Omega$ in Theorem \ref{thm 1}
 ensures the validity of  the orthogonality conditions \eqref{3.2}, as  Gaussian  measure is  radially
 symmetric about the origin.
It remains an interesting question that whether Ashbaugh and Benguria’s conjecture  \eqref{1.1} 
holds in Gauss space.

The rest of the paper is organized as follows.
In Section \ref{sect2}, we study properties of the first nonzero eigenvalue and eigenfunctions for balls in Gauss space. In Section  \ref{sect3}, we prove Theorem \ref{thm 1}.

\section{Eigenvalue Problem for Balls In Gauss Space}\label{sect2}
In this section, we  prove some properties of the first nonzero Neumann eigenvalue and its eigenfunctions  for round balls in Gauss space. Let $B_{R}\subset\R^m$ denote the origin-centered  ball   with radius $R$. The Neumann eigenvalue problem on $B_R$  is  
\begin{align}\label{2.1}
\begin{cases}
    -\Delta u+x\cdot \nabla u=\mu(B_R) u \quad & \text{in $B_R$,}\\
\frac{\p u}{\p \nu}=0 \quad & \text{on $\p B_R$,}
\end{cases}
\end{align}
where $\nu$ is the unit outer normal  on $\p B_R$.
From Lemma 4.1 of \cite{CB12}, the first nonzero eigenvalue of \eqref{2.1} has multiplicity $m$, meaning
$$
\mu_1(B_R)=\mu_2(B_R)=\cdots=\mu_m(B_R),
$$
and the corresponding eigenfunctions  are of the form $u_i(x)=g(r)\psi_i(\theta)$, where $(r, \theta)$ are  polar coordinates,
$\psi_i(\theta)$ are the linear coordinate functions restricted to $\mathbb{S}^{m-1}$, and $g(r)$ satisfies
\begin{equation}\label{2.2}
g''(r)+\Big(\frac{m-1}{r}-r\Big) g'(r)+\Big(\mu_1(B_R)-\frac{m-1}{r^2}\Big) g(r)=0, \quad r\in (0,R)
\end{equation}
subject to the boundary conditions $g(0)=0$ and $g'(R)=0$.

Multiplying \eqref{2.2} by $g(r)r^{m-1}e^{-r^2/2}$ and integrating over $[0,R]$, we obtain 
\begin{align}\label{2.3}
\mu_1(B_R)
=\frac{\int_{0}^R\Big(g'(r)^2+\frac{m-1}{r^2}g(r)^2\Big)e^{-\frac{r^2}{2}}r^{m-1}\, dr}{\int_{0}^R g(r)^2e^{-\frac{r^2}{2}}r^{m-1}\, dr}.
\end{align}
Moreover, the first nonzero eigenvalue of \eqref{2.1} has the following the variational characterization 
\begin{align}\label{2.4}
\mu_1(B_R)=\inf_{\vp\in C^1}\Big\{\frac{\int_{0}^R\Big(\vp'(r)^2+\frac{m-1}{r^2}\vp(r)^2\Big)e^{-\frac{r^2}{2}}r^{m-1}\, dr}{\int_{0}^R \vp(r)^2e^{-\frac{r^2}{2}}r^{m-1}\, dr}: \vp(0)=0, \vp'(R)=0\Big\}.
\end{align}
From  \eqref{2.4}, it is clear that $g(r)$ does not change sign in $[0,R]$; thus we assume $g(r)\ge 0$ in what follows.

\begin{lemma}\label{Lemma 2.1}
 $\mu_1(B_R)$ is strictly decreasing in $R$. 
\end{lemma}
\begin{proof}
The monotonicity of  $\mu_1(B_R)$ with respect to $R$ holds for all regular Sturm-Liouville problems. For completeness, we provide a proof via the variational characterization \eqref{2.4} of $\mu_1(B_R)$.

For $0<a<b$, let $g_{a}(r)$ be an eigenfunction corresponding to $\mu_1(B_{a})$, and extend $g_{a}$ to  $[0, b]$ by  
$$
g_b(r)=\begin{cases}
    g_a(r), \quad& r\le a,\\
    g_a(a), \quad& a<r\le b.
\end{cases}
$$
Clearly, $g_b(r)$ is admissible for $\mu_1(B_b)$, so
\begin{align}\label{2.5}
\mu_1(B_b)\le&  \frac{\int_{0}^b\Big(g_b'(r)^2+\frac{m-1}{r^2}g_b(r)^2\Big)e^{-\frac{r^2}{2}}r^{m-1}\, dr}{\int_{0}^b g_b(r)^2e^{-\frac{r^2}{2}}r^{m-1}\, dr}\\
  =&\frac{\int_{0}^a\Big(g_a'(r)^2+\frac{m-1}{r^2}g_a(r)^2\Big)e^{-\frac{r^2}{2}}r^{m-1}\, dr+\int_{a}^b\frac{m-1}{r^2}g_a(a)^2e^{-\frac{r^2}{2}}r^{m-1}\, dr}{\int_{0}^a g_a(r)^2e^{-\frac{r^2}{2}}r^{m-1}\, dr+\int_{a}^b g_a(a)^2e^{-\frac{r^2}{2}}r^{m-1}\, dr}.\nonumber
\end{align}
Note that
$$
\frac{\int_{a}^b\frac{m-1}{r^2}e^{-\frac{r^2}{2}}r^{m-1}\, dr}{\int_{a}^be^{-\frac{r^2}{2}}r^{m-1}\, dr}<\frac{m-1}{a^2}<\frac{\int_{0}^a\frac{m-1}{r^2}g_a(r)^2e^{-\frac{r^2}{2}}r^{m-1}\, dr}{\int_{0}^ag_a(r)^2e^{-\frac{r^2}{2}}r^{m-1}\, dr},
$$
we obtain
\begin{align*}
& \frac{\int_{a}^b\frac{m-1}{r^2}g_a(a)^2e^{-\frac{r^2}{2}}r^{m-1}\, dr+\int_{0}^a\frac{m-1}{r^2}g_a(r)^2e^{-\frac{r^2}{2}}r^{m-1}\, dr}{\int_{a}^bg_a(a)^2e^{-\frac{r^2}{2}}r^{m-1}\, dr+\int_{0}^ag_a(r)^2e^{-\frac{r^2}{2}}r^{m-1}\, dr} \\  
<\quad&\frac{\int_{0}^a\frac{m-1}{r^2}g_a(r)^2e^{-\frac{r^2}{2}}r^{m-1}\, dr}{\int_{0}^ag_a(r)^2e^{-\frac{r^2}{2}}r^{m-1}\, dr}.
\end{align*}
Substituting above inequality into \eqref{2.5} yields
\begin{align*}
     \mu_1(B_b)< \frac{\int_{0}^a\Big(g_a'(r)^2+\frac{m-1}{r^2}g_a(r)^2\Big)e^{-\frac{r^2}{2}}r^{m-1}\, dr}{\int_{0}^a g_a(r)^2e^{-\frac{r^2}{2}}r^{m-1}\, dr}=\mu_1(B_a),
\end{align*}
as required.
\end{proof}

\begin{lemma}\label{Lemma 2.2}
For any $R>0$, we have
\begin{align}\label{2.6}
    \mu_1(B_R)>1.
\end{align}
\end{lemma}
\begin{proof}
This follows from the equality case of  Poincar\'e-Wirtinger inequality, proved in \cite[Theorem 1.1]{BCHT13}(see also \cite{BJ21,BCKT16}). 
\end{proof}
\begin{remark}
In fact, the conclusion of  Lemma \ref{Lemma 2.2}  holds true for all convex bounded domains (not just for the ball), see  \cite{BCHT13, BCKT16} and the references therein. Moreover if $\Omega$
 is a convex
domain in  $\R^n$ then $\mu_1(\Omega)=1$ if and only if $\Omega$
 is a strip, see  \cite{BJ21} and the references therein.
\end{remark}
\begin{lemma}\label{Lemma 2.3}
Let $g(r)$, $r\in[0, R]$, be a nonnegative eigenfunction corresponding to $\mu_1(B_R)$ defined by \eqref{2.2}. Then 
\begin{align}\label{2.7}
    g'(r)>0, \quad\quad r\in(0,R),
\end{align}
and 
\begin{align}\label{2.8}
   g'(r)-\frac{ g(r)}{r}\le 0, \quad\quad r\in(0,R].
\end{align}
\end{lemma}
\begin{proof} 
Assume by contradiction that there exists  $r\in(0, R)$ such that $g'(r)=0$. Recall from  \eqref{2.2} that
\begin{align}\label{2.9}
    g''(t)+\Big(\frac{m-1}{t}-t\Big) g'(t)+\Big(\mu_1(B_R)-\frac{m-1}{t^2}\Big) g(t)=0,\quad \quad t\in (0,r) 
\end{align}
subject to $g(0)=0$ and $g'(r)=0$. Taking $g(t)$ ($t\in(0, r)$) as a trial function of $\mu_1(B_r)$ and using  \eqref{2.9}, we find:
\begin{align*}
    \mu_1(B_r)\le \frac{\int_{0}^r\Big(g'(t)^2+\frac{m-1}{t^2}g(t)^2\Big)e^{-t^2/2}t^{m-1}\, dt}{\int_{0}^r g(t)^2e^{-t^2/2}t^{m-1}\, dt}=\mu_1(B_R)
\end{align*}
contradicting with Lemma \ref{Lemma 2.1}. Thus $g'(r)\neq 0$ in $(0, R)$, hence \eqref{2.7} holds. 

Now we prove \eqref{2.8}. Set
$$H(r)=g'(r)-\frac{ g(r)}{r}, \qquad\qquad r\in(0, R]. $$ 
Since $g(0)=0$, $g'(R)=0$, and $g'(r)>0$ for $r\in(0, R)$, we have
\begin{align}\label{2.10}
   \lim_{r\to 0^+} H(r)=0, \qquad H(R)=-\frac{g(R)}{R}<0. 
\end{align}

Taking the derivative with respect to $r$, we obtain  
\begin{align}\label{2.11}
    H'(r)=g''(r)+\frac{g(r)}{r^2}-\frac{g'(r)}{r},
\end{align}
Eliminating  $g''(r)$ from \eqref{2.2} and \eqref{2.11},  we have
\begin{align}\label{2.12}
    H'(r)=\Big(r-\frac{m}{r}\Big)g'+\Big(\frac{m}{r^2}-\mu_1(B_R)\Big)g(r).
\end{align}

We claim that 
\begin{align}\label{2.13}
 H(r)\le 0 \quad\text{for}\quad r\in (0, \sqrt{m}), \text{\quad and \quad } H(a)<0 \quad \text{for some} \quad a\in(0, \sqrt{m}).   
\end{align}
 Indeed, if there exists an $r_1\in (0, \sqrt{m})$ such that $H(r_1)\ge 0$, i.e. 
 \begin{align*}
 g'(r_1)\ge \frac{g(r_1)}{r_1},
 \end{align*}
 then plugging above inequality into \eqref{2.12} yields
\begin{align}\label{2.14}
    H'(r_1)=\Big(r_1-\frac{m}{r_1}\Big)g'(r_1)+\Big(\frac{m}{r_1^2}-\mu_1(B_R)\Big)g(r_1)
    \le (1-\mu_1(B_R))g(r_1)
    <0,
\end{align}
where we used  $\mu_1(B_R)>1$ (see Lemma \ref{Lemma 2.2}) in the inequality. Thus we deduce from \eqref{2.10} and   \eqref{2.14} that $H(r)\le 0$ for $r\in(0,\sqrt{m})$. If $H(r)\equiv 0$ for all $r\in (0, \sqrt{m})$, then $g(r)=g'(0) r$ in $(0, \sqrt{m})$, contradicting with \eqref{2.2}. Thus Claim \eqref{2.13} comes true.

Now we prove $H(r)\le 0$ for all $r\in(a, R)$. Assume  $H(r_2)=0$ for some $r_2\in(0,R)$, namely 
\begin{align*}
g'(r_2)=\frac{g(r_2)}{r_2},
\end{align*}
combined with \eqref{2.12}, if follows
\begin{align}\label{2.15}
    H'(r_2)=\Big(r_2-\frac{m}{r_2}\Big)g'+\Big(\frac{m}{r_2^2}-\mu_1(B_R)\Big)g(r_2)
    =(1-\mu_1(B_R))g(r_2)
    <0,
\end{align}
Therefore we conclude from \eqref{2.10}, \eqref{2.13} and \eqref{2.15} that 
 $H(r)\le 0$ for all $r\in [a, R]$.
Hence we complete the proof of \eqref{2.8}.
\end{proof}	

\section{Proof of Main Theorem}\label{sect3}
In this section, we prove  Theorem \ref{thm 1}. The main idea is to  construct trial functions for $\mu_k$ (see \eqref{1.4}) using techniques introduced by Weinberger \cite{Wei56} and further developed by Xia and Wang in \cite{XW23}. We first recall a monotonicity lemma for Gaussian symmetrization.
\begin{lemma}\label{Lemma 3.1}
    Let $\Omega\subset \R^m$ be an open set, and let $B$ be the origin-centered round ball with the same Gaussian volume as $\Omega$, i.e. $\int_B \, d\gamma_m=\int_\Omega \, d\gamma_m$. If $h(r)$ is nonincreasing  on  $[0,+\infty)$, then 
    \begin{align*}
        \int_\Omega h(|x|)\, d\gamma_m\le    \int_B h(|x|)\, d\gamma_m.
    \end{align*}
Here, $d\gamma_m=(2\pi)^{-m/2}e^{-|x|^2/2}dx$ is the $m$-dimensional Gaussian measure.
\end{lemma}
\begin{proof}
    See the proof of  inequality (4.36) in \cite[Page 213]{CB12}.
\end{proof}
\begin{proof}[Proof of Theorem \ref{thm 1}]
  Let $u_i(x)$ be an eigenfunction corresponding to $\mu_i(\Omega)$. The Neumann eigenvalues satisfy the variational principle
    \begin{align}\label{3.1}
        \mu_i(\Omega)=\inf\Big\{\frac{\int_\Omega |\nabla u|^2\, d\gamma_m}{\int_\Omega u^2\, d\gamma_m}: 
        \int_\Omega u u_k \, d\gamma_m=0 \text{ for } 0\le k<i\Big\}.
    \end{align}
Let $B\subset\R^m$ be the origin-centered ball with radius  $R$ and the same Gaussian volume as $\Omega$.
Define  $G(r):[0, \infty) \to [0, \infty)$ as 
        \begin{align*}
        G(r)=\begin{cases}
            g(r),&\quad r<R,\\
            g(R),&\quad r\geq R,
        \end{cases}
    \end{align*}
    where $g(r)$ is defined in Section \ref{sect2}. For $1\le i\le m$,  define
    \begin{align*}
        v_i(x)=G(|x|)\frac{x_i}{|x|}.
    \end{align*}
Since  $\Omega$ is symmetric about the origin, we have
     \begin{align*}
        \int_\Omega v_i(x)   d\gamma_m=0, \quad i=1,2,\cdots, m.
    \end{align*}
For $1\le i,  j\le m$, let
    \begin{align*}
    q_{ij}= \int_{\Omega} v_i(x) u_j(x) \, d\gamma_m,\end{align*}.
By  QR-factorization, there exists  an orthogonal matrix $A=(a_{ij})$ such that 
    \begin{align*}
        0=\int_{\Omega} \sum_{k=1}^m a_{ik} v_k(x)  u_j(x)\, d\gamma_m
    \end{align*}
    for all $1\le j< i\le m$.
Thus, by appropriately selecting the coordinate axes, we may assume that
    \begin{align}\label{3.2}
        \int_{\Omega} v_i(x)u_j(x)\, d\gamma_m=0
    \end{align}
    for $1\le j<i\le m$. 
    
According to \eqref{3.2} and variational formulation \eqref{3.1} of $\mu_i(\Omega)$, $v_i(x)$ is admissible for  $\mu_i(\Omega)$, i.e. for $1\le i\le m$,
    \begin{align}\label{3.3}
        \int_{\Omega} v_i(x)^2 \, d\gamma_m \leq \frac{1}{\mu_i(\Omega) } \int_{\Omega} |\nabla v_i(x) |^2\, d\gamma_m
    \end{align}
A direct computation shows
\begin{equation*} 
        |\nabla v_i(x)|^2 = G'(r)^2 \frac{x_i^2}{|x|^2} + G(r)^2 (\frac 1 {|x|^2}-\frac{x_i^2}{|x|^4}).
\end{equation*}
Substituting this into  \eqref{3.3} and summing over $i$, we obtain
    \begin{align}\label{3.4}
\int_{\Omega} G(r)^2e^{-\frac{|x|^2}2}\, dx
\leq&  \sum_{i=1}^m \frac{1}{\mu_i(\Omega)} \int_{\Omega}  G'(r)^2 \frac{x_i^2}{|x|^2} e^{-\frac{|x|^2}2}\, dx  \\
& +\sum_{i=1}^m \frac{1}{\mu_i(\Omega)} \int_{\Omega} G(r)^2 (\frac 1 {|x|^2}-\frac{x_i^2}{|x|^4})e^{-\frac{|x|^2}2}\, dx.\nonumber
    \end{align}
Note that  $G(r)$ is a constant for $r>R$, so 
    \begin{align} \label{3.5}
\int_{\Omega}  G'(r)^2 \frac{x_i^2}{|x|^2}  e^{-|x|^2/2}\, dx 
= &\int_{\Omega \cap B} G'(r)^2 \frac{x_i^2}{|x|^2}   e^{-|x|^2/2}\, dx\\
       \leq  & \int_{B} G'(r)^2 \frac{x_i^2}{|x|^2} e^{-|x|^2/2}\, dx \nonumber\\
        = & \frac{1}{m} \int_{B} G'(r)^2 e^{-|x|^2/2}\, dx.\nonumber
    \end{align}
Using the ordering $\mu_1(\Omega)\le \mu_2(\Omega)\le \cdots\le \mu_m(\Omega)$ and the identity $\sum_{i=1}^m \frac{x_i^2}{|x|^2}=1$, we estimate  
    \begin{align}\label{3.6}
   &\sum_{i=1}^m \frac{1}{\mu_i(\Omega)} \int_{\Omega} G(r)^2 (\frac 1 {|x|^2}-\frac{x_i^2}{|x|^4}) e^{-|x|^2/2}\, dx\\
 =&   \sum_{i=1}^{m-1} \frac{1}{\mu_i(\Omega)} \int_{\Omega} G(r)^2 (\frac 1 {|x|^2}-\frac{x_i^2}{|x|^4}) e^{-|x|^2/2}\, dx+\frac{1}{\mu_m(\Omega)} \sum_{i=1}^{m-1} \int_{\Omega} G(r)^2 \frac{x_i^2}{|x|^4} e^{-|x|^2/2}\, dx\nonumber\\
\le &\sum_{i=1}^{m-1} \frac{1}{\mu_i(\Omega)} \int_{\Omega} G(r)^2 \frac 1 {|x|^2} e^{-|x|^2/2}\, dx.\nonumber     \end{align}
Since $G(r)/r$ is nonincreasing in $(0, \infty)$ by Lemma \ref{Lemma 2.3}, it follows from  Lemma \ref{Lemma 3.1} that
\begin{align}\label{3.7}
    \int_{\Omega} \frac{G(r)^2}{|x|^2}  e^{-|x|^2/2}\, dx\le \int_{B} \frac{G(r)^2}{|x|^2}  e^{-|x|^2/2}\, dx.
\end{align}
Combining  \eqref{3.6} and  \eqref{3.7}, we have
\begin{align}\label{3.8}
   \sum_{i=1}^m \frac{1}{\mu_i(\Omega)} \int_{\Omega} G(r)^2 (\frac 1 {|x|^2}-\frac{x_i^2}{|x|^4}) e^{-|x|^2/2}\, dx
\le &\sum_{i=1}^{m-1} \frac{1}{\mu_i(\Omega)}\int_{B} \frac{G(r)^2}{|x|^2}  e^{-|x|^2/2}\, dx.  
    \end{align}
Substituting \eqref{3.5} and  \eqref{3.8} into \eqref{3.4} yields
\begin{align}\label{3.9}
&\int_\Omega G(r)^2 e^{-\frac {|x|^2} 2}\, dx\\
\le& \sum_{i=1}^m \frac 1{m\mu_i(\Omega)}\int_{B} G'(r)^2 e^{-|x|^2/2}\, dx+\sum_{i=1}^{m-1} \frac{1}{\mu_i(\Omega)}\int_{B} \frac{G(r)^2}{|x|^2}  e^{-|x|^2/2}\, dx\nonumber\\
\le& \sum_{i=1}^{m-1} \frac 1{(m-1)\mu_i(\Omega)}\int_{B} G'(r)^2 e^{-|x|^2/2}\, dx+\sum_{i=1}^{m-1} \frac{1}{\mu_i(\Omega)}\int_{B} \frac{G(r)^2}{|x|^2}  e^{-|x|^2/2}\, dx\nonumber\\
=& \sum_{i=1}^{m-1} \frac 1{(m-1)\mu_i(\Omega)}\int_{B} \big(G'(r)^2+(m-1)\frac{G(r)^2}{r^2}\big)e^{-|x|^2/2}\, dx,\nonumber
\end{align}
where the second inequality uses $\mu_i(\Omega)\le \mu_m(\Omega)$ for $1\le i\le m-1$.

Similarly, since $G(r)$ is nondecreasing in $(0, \infty)$, Lemma \ref{Lemma 3.1} implies
\begin{align}\label{3.10}
    \int_{\Omega} G(r)^2  e^{-|x|^2/2}\, dx\ge \int_{B} G(r)^2  e^{-|x|^2/2}\, dx.
\end{align}
Combining \eqref{3.9} and \eqref{3.10}, we get
\begin{align*}
 \int_{B} G(r)^2  e^{-|x|^2/2}\, dx\le  \sum_{i=1}^{m-1} \frac{1}{(m-1) \mu_i(\Omega) }\int_{B} \Big(G'(r)^2+\frac{m-1}{r^2} G(r)^2\Big)  e^{-|x|^2/2}\, dx,   
\end{align*}
which gives
    \begin{align*}
         \sum_{i=1}^{m-1} \frac{1}{\mu_{i}(\Omega) }
         \ge&   \frac{	(m-1) \int_{B} G(r)^2e^{-|x|^2/2}\, dx }{\int_{B} \big(G'(r)^2 +(m-1)G(r)^2 /r^2\big)e^{-|x|^2/2}\, dx } \\
         =& \frac{(m-1)\int_{0}^R g(r)^2e^{-\frac{r^2}{2}}r^{m-1}\, dr}{\int_{0}^R\Big(g'(r)^2+\frac{m-1}{r^2}g(r)^2\Big)e^{-\frac{r^2}{2}}r^{m-1}\, dr} \\ 
         =&\frac{m-1}{\mu_1(B)},
    \end{align*}
     where  the last equality follows from \eqref{2.3}. 
     
If equality holds in the above inequality,
    then equality in  \eqref{3.5} implies $B\setminus \Omega=\emptyset$.
Since $\Omega$ and $B$ have the same Gaussian volume, it follows that $\Omega=B$. 
     This completes the proof.
\end{proof}

\section*{Acknowledgments}
The authors would like to thank 
Professor Rafael  Benguria for his interest in this paper, 
and thank Professor Francesco Chiacchio  for his helpful discussion regarding Remark 4.3 of \cite{CB12}, as well as for bringing references \cite{BJ21, BCHT13, BCKT16} to our attention. We also thank Professor Jing Mao for explaining his paper with Chen \cite{CM24}. The first author is supported  by Undergraduate Training Program for  Innovation and Entrepreneurship, Soochow University. The second author is supported by NSF of Jiangsu Province No. BK20231309. We thank the anonymous referee for valuable comments on the manuscript.

 \bibliographystyle{plain}
	\bibliography{ref}

\end{document}